\documentclass[12pt]{amsart}
\usepackage{graphicx, hyperref}
\title{Odd partitions in Young's lattice}
\author{Arvind Ayyer}
\address{AA: Department of Mathematics, Indian Institute of Science, Bengaluru 560012, India.}
\email{arvind@math.iisc.ernet.in}
\author{Amritanshu Prasad}
\address{AP: The Institute of Mathematical Sciences, CIT campus, Taramani Chennai 600113, India.}
\email{amri@imsc.res.in}
\author{Steven Spallone}
\address{SS: Indian Institute of Science Education and Research, Pashan, Pune 411008, India.}
\email{sspallone@iiserpune.ac.in}
\keywords{Partitions, Young's lattice, cores, quotients, hooks, odd-dimensional representations, symmetric groups, Macdonald tree, differential posets}
\subjclass[2010]{05A15, 05A17, 05E10, 20C30}

\newtheorem{prop}{Proposition}
\newtheorem{lemma}{Lemma}

\newtheorem{theorem}{Theorem}

\newcommand{\core}[2]{\corem_{#2}(#1)}
\newcommand{\quot}[2]{\quo_{#2}(#1)}

\DeclareMathOperator{\quo}{quo}
\DeclareMathOperator{\corem}{core}
\DeclareMathOperator{\hand}{hand}
\DeclareMathOperator{\foot}{foot}

\begin{document}
\begin{abstract}
  We show that the subgraph induced in Young's graph by the set of partitions with an odd number of standard Young tableaux is a binary tree.
  This tree exhibits self-similarities at all scales, and has a simple recursive description.
\end{abstract}
\maketitle

\thispagestyle{myheadings}
\font\rms=cmr8 
\font\its=cmti8 
\font\bfs=cmbx8

\markright{\its S\'eminaire Lotharingien de
Combinatoire \bfs 75 \rms (2016), Article~B75g\hfill}
\def\thepage{}

\section{Introduction}
\label{sec:introduction}
Young's lattice is the set $\Lambda$ of integer partitions, partially ordered by containment of Young diagrams.
It has a unique minimal element $\emptyset$, the trivial partition of $0$.
Its Hasse diagram is known as Young's graph.
For each $\lambda\in \Lambda$, let $f_\lambda$ denote the number of saturated chains from $\emptyset$ to  $\lambda$.
This number $f_\lambda$ is also the number of standard tableaux of shape $\lambda$, and the dimension of the irreducible representation of the symmetric group associated with $\lambda$.
It can be computed by the hook-length formula of Frame, Robinson and Thrall \cite[Theorem~1]{frame1954hook}.

Let $\Lambda^{\mathrm{odd}}$ denote the subset of partitions $\lambda\in \Lambda$ for which $f_\lambda$ is odd.
The partitions in $\Lambda^{\mathrm{odd}}$ will be referred to as \emph{odd partitions}.
Macdonald \cite{macdonald1971degrees} has shown that the number of odd partitions of $n$ is $2^{\alpha(n)}$, where,  if $n$ has binary expansion $n = a_0 + 2a_1 + 2^2a_2 + 2^3a_3 +\dotsb$, with $a_i\in \{0, 1\}$, then
\begin{displaymath}
  \alpha(n) = a_1 + 2a_2 + 3a_3 + \dotsb.
\end{displaymath}
In this article, we show (Theorems~\ref{theorem:unique-parent} and~\ref{theorem:children}) that the subgraph induced in Young's graph by $\Lambda^{\mathrm{odd}}$ is an incomplete binary tree (Figure~\ref{fig:macdonald-tree}).
We call this tree the \emph{Macdonald tree}.
For each $\lambda \in \Lambda^{\mathrm{odd}}$, we determine the number of branches of $\lambda$ in this tree.
This tree has self-similarities at all scales (Lemma~\ref{lemma:fractal}), only two infinite rays (Theorem~\ref{theorem:ray}) and a simple recursive description (Section~\ref{sec:recursive}).
Hook-shaped partitions form an order ideal of Young's lattice.
When each partition $\lambda$ in the Hasse diagram of this ideal is replaced by $f_\lambda$, Pascal's triangle is obtained.
The intersection of this ideal with $\Lambda^{\mathrm{odd}}$ is the tree of odd binomial coefficients in Pascal's triangle (see Section~\ref{sec:pascal}).
\begin{figure}
  \centering
  \includegraphics[scale=0.67]{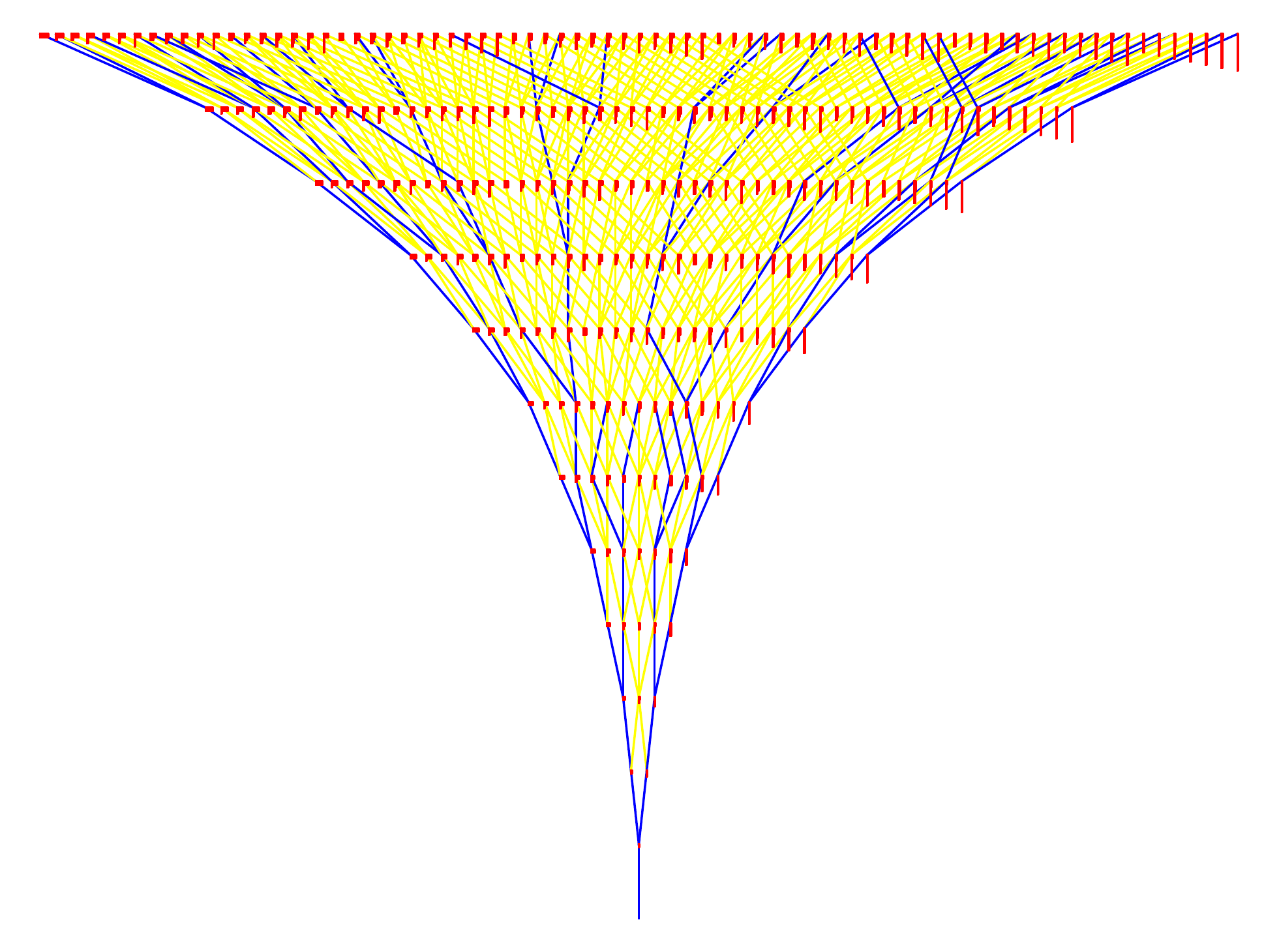}
  \caption{The Macdonald tree (blue edges) in Young's graph}
  \label{fig:macdonald-tree}
\end{figure}

{\bf Note added in proof:} While this paper was being reviewed, its results were applied by Gianelli, Kleshchev, Navarro and Tiep \cite{gagp} to construct an explicit McKay correspondence for symmetric groups.

\section{Preliminary Results}
\label{sec:preliminary-results}
Our first lemma is based on the theory of cores and quotients of partitions.
An exposition of this theory, including definitions concerning hooks and their anatomy, can be found in \cite[Section~2.7]{2009JamesKerber}.
For a partition $\lambda$ and an integer $p>1$, let $\core\lambda p$ denote the $p$-core of $\lambda$ and $\quot\lambda p$ denote the $p$-quotient.
By definition, $\core\lambda p$ is a partition with no hook-length divisible by $p$, and $\quo_p(\lambda)$ is a family $(\mu^0,\dotsc,\mu^{p-1})$ of partitions such that $|\lambda| = |\core\lambda p| + p(|\mu^0| + \dotsb + |\mu^{p-1}|)$.
\begin{lemma}
  \label{lemma:unique_hook}
  Suppose that $2^k \leq n <2^{k+1}$, and $\lambda$ is a partition of $n$.
  Then $\lambda$ is odd if and only if $\lambda$ has a unique $2^k$-hook, and $\core\lambda{2^k}$ is odd.
  Moreover, for each odd partition $\mu$ of $n-2^k$, there are $2^k$ odd partitions $\lambda$ of $n$ with $\core\lambda{2^k}=\mu$.
\end{lemma}
This lemma follows from the discussion in \cite[Section~6]{olsson}.
A self-contained proof is given in Section~\ref{sec:append-result-from} of this article.

\pagenumbering{arabic}
\addtocounter{page}{1}
\markboth{\SMALL ARVIND AYYER, AMRITANSHU PRASAD, AND STEVEN SPALLONE}{\SMALL 
ODD PARTITIONS IN YOUNG'S LATTICE}

If $\lambda$ covers $\mu$ in Young's lattice, we write $\lambda\in \mu^+$, or $\mu\in \lambda^-$.
\begin{lemma}[\sc Main Lemma]
  \label{lemma:hook-removal}
  Suppose that $\lambda$ and $\mu$ are odd partitions with $\mu\in \lambda^-$.
  Assume that $2^k\leq |\mu|$, and $|\lambda|<2^{k+1}$.
  \begin{enumerate}
  \item We have $\core\mu{2^k} \in \core\lambda{2^k}^-$.
  \item
    Let $r$ be the rim of the unique $2^k$-hook of $\lambda$, and $s$ be the rim of the unique $2^k$-hook of $\mu$ (see Lemma~\ref{lemma:unique_hook}).
    Let $c$ be the unique cell of $\lambda$ that is not in $\mu$,
    and $c'$ be the unique cell of $\core\lambda{2^k}$ which is not in $\core\mu{2^k}$, then exactly one of the following holds:
    \begin{enumerate}
    \item $c$ has no neighbour in $r$, $r = s$, and $c' = c$,
    \item $c = \hand(r) = \hand(s)^E$, $c' = \foot(r)^W = \foot(s)$.
    \item $c = \foot(r) = \foot(s)^S$, $c' = \hand(r)^N = \hand(s)$.
    \item $\{c^N, c^W\}\subset r\cap s$, $c'=c^{NW}$.
    \end{enumerate}
    In \emph{(b)-(d)}, $s$ is obtained from $r$ by removing $c$ and adding $c'$.
  \end{enumerate}
\end{lemma}
Here $c^N$, $c^S$, $c^E$, and $c^W$ denote the cells directly to the north, south, east, and west of $c$.
Also, $\hand(r)$ denotes the node $c$ of $r$ for which neither $c^N$ nor $c^E$ is in $r$.
Similarly, $\foot(r)$ denotes the node of $r$ for which neither $c^W$ nor $c^S$ is in $r$.

Before proving Lemma~\ref{lemma:hook-removal}, we formulate two simpler lemmas which will be used in its proof.
\begin{lemma}
  \label{lemma:large-hook-lemma}
  Let $\lambda$ be any partition and $c$ be a cell of $\lambda$ such that its hook $h(c)$ satisfies $|h(c)|\geq |\lambda|/2$.
  Then $c$ lies in the first row or in the first column of $\lambda$.
\end{lemma}
\begin{proof}
  If not, consider the cell $c^{NW}$ which lies to the north-west of $c$.
  The hooks $h(c^{NW})$ and $h(c)$ are disjoint.
  Also, $|h(c^{NW})| \geq |h(c)| + 2$.
  So $|\lambda|\geq |h(c)| + |h(c^{NW})|\geq 2|h(c)|+2\geq |\lambda|+2$, a contradiction.
\end{proof}
\begin{lemma}
  \label{lemma:two-big-hooks}
  If $c_1$ and $c_2$ are two cells in the Young diagram of a partition $\lambda$ such that $|h(c_1)|\geq |\lambda|/2$ and $|h(c_2)|> |\lambda|/2$, then $c_1$ and $c_2$ lie either in the same row or in the same column.
\end{lemma}
\begin{proof}
  If either $c_1$ or $c_2$ is the cell $(1,1)$, then the result follows from Lemma~\ref{lemma:large-hook-lemma}.
  Otherwise, if $c_1$ and $c_2$ do not lie in the same row or column, then the hooks $h(c_1)$ and $h(c_2)$ share at most one cell.
  Also, the cell $(1,1)$ is in neither hook.
  So we have:
  \begin{displaymath}
    |h(c_1)|+|h(c_2)| - 1\leq |\lambda|-1,
  \end{displaymath}
  contradicting
  \begin{displaymath}
    |h(c_1)|+|h(c_2)| > |\lambda|
  \end{displaymath}
  from the hypotheses.
\end{proof}
\begin{proof}[Proof of Lemma~\ref{lemma:hook-removal}]
  Suppose $c$ does not have a neighbour in $r$.
  Then $c$ is a removable cell of $\lambda \setminus r = \core\lambda{2^k}$, and $r$ is a $2^k$-rim hook of $\mu = \lambda \setminus c$.
  So $\core\mu{2^k} = \mu \setminus r = (\lambda \setminus r) \setminus c\in \core\lambda{2^k}^-$, giving the first part of the lemma, and Case~(a) of the second part.

  Now suppose that $c$ has a neighbour in $r$.
  Since $c$ is removable from $\lambda$, $c^E$ and $c^S$ cannot be in $r$.
  But if $c^W$ or $c^N$ is in $r$, then $c$ must also be in $r$, because $r$ is a removable rim hook.

  Let $x$ (respectively $y$) denote the node of $\lambda$ for which $h(x)=r$ (respectively $h(y)=s$).
  We may rule out $x=(1,1)$, because then the longest hook of $\mu$ would be of strictly smaller length.

  Suppose $c^W\in r$, but $c^N\notin r$.
  Then $c = \hand(r)$.
  Note that the hook-lengths of $\mu$ are the same as the hook-lengths of $\lambda$ except in the row and column of $c$, where they decrease by one.
  Since $\lambda$ has no other $2^k$-hook, $y$ must lie either in the row of $c$ or in the column of $c$.
  But $y$ cannot lie in the column of $c$, for then it would not lie in the same row or column as $x$, contradicting Lemma~\ref{lemma:two-big-hooks}.
  So $y$ must lie in the row of $c$, which is also the row of $x$.
  This would imply that $\hand(s) = c^W$, and so $\foot(r)^W = \foot(s)$.
  Then $c' = \foot(r)^W$.
  We have $\core\mu{2^k} = \core\lambda{2^k} \setminus c'$, giving the first part of the lemma and Case~(b) of the second part.
  The case where $c^N\in r$, but $c^W$ not in $r$ can be dealt with similarly, and leads to Case~(c).

  Finally, suppose $c^W$ and $c^N$ are both in $r$.
  Then replacing $c$ by $c^{NW}$ in $r$ results in a $2^k$-hook of $\mu$; this must be $s$, giving the first part of the lemma and Case~(d) of the second part.
\end{proof}

By Lemma~\ref{lemma:unique_hook}, the odd partitions of $2^k$ are precisely the hook-shaped ones.
In general, if $2^k\leq n<2^{k+1}$, the function $\mathrm{core}_{2^k}:\lambda\mapsto \core\lambda{2^k}$ maps odd partitions of $n$ onto odd partitions of $n-2^k$.
\begin{prop}
  \label{prop:coring-tree}
  Let $\lambda$ be an odd partition of $n$.
  \begin{enumerate}
  \item If $2^k< n< 2^{k+1}$, then the map
    \begin{displaymath}
      \mathrm{core}_{2^k}:\lambda^-\cap \Lambda^{\mathrm{odd}}_{n-1}\to \core\lambda{2^k}^-\cap\Lambda^{\mathrm{odd}}_{n-2^k-1}
    \end{displaymath}
    is injective.
  \item If $2^k-1<n< 2^{k+1}-1$, then the map
    \begin{displaymath}
      \mathrm{core}_{2^k}:\lambda^+\cap \Lambda^{\mathrm{odd}}_{n+1}\to \core\lambda{2^k}^+\cap\Lambda^{\mathrm{odd}}_{n-2^k+1}
    \end{displaymath}
    is injective.
  \end{enumerate}
\end{prop}
\begin{proof}
  Suppose that $2^k<n< 2^{k+1}$.
  Let $\mu$ and $\nu$ be distinct elements of $\lambda^-\cap \Lambda^{\mathrm{odd}}_{n-1}$.
  Let $c$ and $d$ be the cells of $\lambda$ that are not in $\mu$ and not in $\nu$ respectively, and $c'$ and $d'$ be the cells of $\core\lambda{2^k}$ that are not in $\core\mu{2^k}$ and $\core\nu{2^k}$ respectively.
  Since $\mu\neq \nu$, $c\neq d$.
  We need to show that $c'\neq d'$.

  Consider first the case where $c$ has no neighbour in $r$.
  Then, if $d$ also has no neighbour in $r$, Lemma~\ref{lemma:hook-removal} tells us that $c'=c$ and $d'=d$, so $c'\neq d'$.
  On the other hand, when $d$ satisfies one of cases (b), (c), and (d) of Lemma~\ref{lemma:hook-removal}, then $d$ lies in $r$, and $d'$ has a neighbour in $r$.
  Since $c'=c$ does not have a neighbour in $r$, $c'\neq d'$.

  Now suppose $c=\hand(r)$.
  Then $c'=\foot(r)^W$.
  Since $d\neq c$, $d$ corresponds to one of the cases (a), (c), and (d) in Lemma~\ref{lemma:hook-removal}.
  But none of these can give rise to $c'$, so $d'\neq c'$.
  The remaining cases are similar.
  This concludes the proof of Part (1) of the proposition.
  The proof of Part (2) is similar, but the roles played by $\lambda$ and $\mu$ in Lemma~\ref{lemma:hook-removal} are interchanged.
\end{proof}

\section{Tree structure}
\label{sec:tree-structure}
\begin{theorem}
  [\sc Unique Parent Theorem]
  \label{theorem:unique-parent}
  For every odd partition $\lambda$ with $|\lambda|>1$, there exists a unique odd partition $\mu\in \lambda^-$.
In other words, the subgraph induced in Young's graph by $\Lambda^{\mathrm{odd}}$ is a rooted tree.
\end{theorem}
\begin{proof}
  Let $n = |\lambda|$, and let $\nu(n)$ denote the sum of the binary digits of $n$.
  We proceed by induction on $\nu(n)$.
  If $\nu(n)=1$, then $n=2^k$.
  In this case, $\lambda$ is a hook, say $(r, 1^{2^k-r})$, and $f_\lambda = \binom{2^k-1}{r-1}$ (see Section~\ref{sec:pascal}).
  Recall the well-known result (see e.g., \cite[Exercise~1.14]{ec1}) that a binomial coefficient $\binom nm$ is odd if and only if the sets of place values where $1$ occurs in the binary expansions of $m$ and $n-m$ are disjoint (in other words, there are no carries when $m$ and $n-m$ are added in binary).
  The set $\lambda^-$ consists of the partitions $(r-1, 1^{2^k-r})$ and $(r, 1^{2^k-r-1})$ whose dimensions are $\binom{2^k-2}{r-2}$ and $\binom{2^k-2}{r-1}$.
  The former is odd when $r$ is even, and the latter is odd when $r$ is odd.
  In any case, one of them is odd, and the other is even, proving the theorem for $\nu(n) =1$.

  If $\nu(n)>1$, then $2^k<n<2^{k+1}$ for some $k\geq 1$.
  Since $\nu(n-2^k)<\nu(n)$, by induction, we may assume that $|\core\lambda{2^k}^-\cap \Lambda_{n-2^k-1}^{\mathrm{odd}}| = 1$.
  Proposition~\ref{prop:coring-tree} implies that
  \begin{displaymath}
    |\lambda^-\cap \Lambda_{n-1}^{\mathrm{odd}}| \leq 1.
  \end{displaymath}
  By Pieri's rule, we have
  \begin{displaymath}
    f_\lambda = \sum_{\mu\in\lambda^-} f_\mu.
  \end{displaymath}
  So if $f_\lambda$ is odd, then $f_\mu$ is odd for at least one $\mu\in \lambda^-$. In other words,
  \begin{displaymath}
    |\lambda^-\cap \Lambda_{n-1}^{\mathrm{odd}}| \geq 1,
  \end{displaymath}
  hence equality holds.
\end{proof}
\begin{prop}
  \label{prop:counting-by-core}
  Suppose $2^k-1<n<2^{k+1}-1$ for some positive integer $k$.
  Then, for any odd partition $\lambda$ of $n$,
  \begin{displaymath}
    |\lambda^+\cap \Lambda^{\mathrm{odd}}_{n+1}| = |\core\lambda{2^k}^+\cap \Lambda^{\mathrm{odd}}_{n-2^k+1}|.
  \end{displaymath}
\end{prop}
\begin{proof}
  We have
  \begin{align*}
    2^{\alpha(n+1)} & = |\Lambda_{n+1}^{\mathrm{odd}}|\\
    & = \sum_{\lambda\in \Lambda_n^\mathrm{odd}} |\lambda^+\cap \Lambda_{n+1}^{\mathrm{odd}}| & \text{(by Theorem~\ref{theorem:unique-parent})}\\
    & = \sum_{\mu\in \Lambda_{n-2^k}^{\mathrm{odd}}}\sum_{\core\lambda{2^k} = \mu} |\lambda^+\cap \Lambda_{n+1}^{\mathrm{odd}}| \\
    & \leq \sum_{\mu\in \Lambda_{n-2^k}^{\mathrm{odd}}}\sum_{\core\lambda{2^k} = \mu} |\mu^+\cap \Lambda_{n-2^k+1}^{\mathrm{odd}}| & \text{(by Proposition~\ref{prop:coring-tree})}\\
    & = 2^k \sum_{\mu\in \Lambda_{n-2^k}^{\mathrm{odd}}}|\mu^+\cap \Lambda_{n-2^k+1}^{\mathrm{odd}}| & \text{(by Lemma~\ref{lemma:unique_hook})}\\
    & = 2^k |\Lambda^{\mathrm{odd}}_{n-2^k-1}| & \text{(by Theorem~\ref{theorem:unique-parent})}\\
    & = 2^k\times 2^{\alpha(n-2^k+1)}\\
    & = 2^{\alpha(n+1)}.
  \end{align*}
  Since the first and last terms are equal, equality holds at each step, and the proposition follows.
\end{proof}
Let $v_2(n)$ denote the $2$-adic valuation of $n$.
\begin{theorem}
  \label{theorem:children}
  Let $\lambda$ be an odd partition of $n$.
  If $n$ is even, then there exists a unique odd partition $\mu\in \lambda^+$.
  If $n$ is odd and $v_2(n+1) = v$, then
  \begin{displaymath}
    |\lambda^+\cap \Lambda_{n+1}^{\mathrm{odd}}| =
    \begin{cases}
      2, & \text{if $\core\lambda{2^v}$ is a hook,}\\
      0, & \text{otherwise}.
    \end{cases}
  \end{displaymath}
  In particular, the induced subgraph of Young's graph consisting of partitions in $\Lambda^{\mathrm{odd}}$ is an incomplete binary tree.
\end{theorem}
\begin{proof}
  When $n$ is even, the theorem is proved by induction on $\nu(n)$.
  If $\nu(n) = 0$, then $n=0$, and the result is obviously true.

  If $n$ is even and $\nu(n)>0$, then $2^k-1<n<2^{k+1}-1$ for some $k>0$.
  Proposition~\ref{prop:counting-by-core} reduces the result to $\core\lambda{2^k}$ of size $n-2^k$.
  Since $\nu(n-2^k)<\nu(n)$, the theorem follows by induction.

  When $n$ is odd, the theorem is proved by induction on $\nu(n+1)$.
  If $\nu(n+1) = 1$ then $n=2^v-1$, and $\core\lambda{2^v} = \lambda$.
  If $\lambda$ is not a hook, then no element of $\lambda^+$ is a hook.
  If $\lambda$ is a hook, then $\lambda^+$ has two hooks.
  Since a partition of $2^v$ is odd if and only if it is a hook, the theorem holds for $\nu(n+1)=1$.

  If $n$ is odd and $\nu(n+1)>1$, then $2^k-1<n<2^{k+1}-1$ for some $k>0$.
  Proposition~\ref{prop:counting-by-core} then reduces the result to $\core\lambda{2^k}$, a partition of $n-2^k$, and, as before, the theorem follows by induction on $\nu(n+1)$.
\end{proof}

The subtrees consisting of the first $2^k-1$ rows, for each $k$, are repeated infinitely many times in the Macdonald tree.
Given a partition $\lambda\in \Lambda^{\mathrm{odd}}$, define $\lambda^{+[0, k]}$ to be the induced subtree rooted at $\lambda$ consisting of nodes of $\Lambda^{\mathrm{odd}}$ which descend from $\lambda$ (more precisely, nodes which are greater than or equal to $\lambda$ in the containment order) and whose ranks lie between $|\lambda|$ and $|\lambda|+k$.
\begin{theorem}
  [\sc Self-similarities of the Macdonald tree]
  \label{lemma:fractal}
  Let $n$ be a positive integer such that $v_2(n)\geq v$.
  Let $\lambda$ be an odd  partition of $n$.
  Then
  \begin{displaymath}
    \mathrm{core}_{2^v}: \lambda^{+[0,2^v-1]}\to \emptyset^{+[0,2^v-1]}
  \end{displaymath}
  is an isomorphism of trees.
\end{theorem}
\begin{proof}
  Suppose that $2^k\leq n<2^{k+1}$ for some $k\geq v$.
  Then, by Proposition~\ref{prop:coring-tree}, the map $\mu\mapsto \core\mu{2^k}$ gives rise to an isomorphism $\lambda^{+[0,2^v-1]}\to \core\lambda{2^k}^{+[0,2^v-1]}$.
  Repeating this operation until $k$ is reduced to $0$, and noting that $\mathrm{core}_{2^k}\circ \mathrm{core}_{2^l} = \mathrm{core}_{2^k}$ for all $k\leq l$, we obtain the desired result.
\end{proof}

By a ray in the Macdonald tree, we mean a sequence $\{\lambda^{(n)}\}_{n=0}^\infty$ of odd partitions such that $\lambda^{(0)}=\emptyset$ and $\lambda^{(n+1)}\in \lambda^{(n)+}$.
\begin{theorem}
  \label{theorem:ray}
  The only rays in the Macdonald tree are $\{(n)\}_{n=1}^\infty$ and $\{(1^n)\}_{n=1}^\infty$.
\end{theorem}
\begin{proof}
  The only hooks in $\Lambda^{\mathrm{odd}}_{2^k+1}$ are $(2^k+1)$ and $(1^{2^k+1})$.
  So if $\lambda\in \Lambda^{\mathrm{odd}}_{2^k}$ is different from $(2^k)$ or $(1^{2^k})$, then $\lambda^+\cap \Lambda^{\mathrm{odd}}_{2^k+1}$ has no hooks in it.
  It follows that none of the partitions of $2^{k+1}-1$ in $\lambda^{+[0, 2^k-1]}$ are hooks, and so, by Theorem~\ref{theorem:children}, have no children.
  Thus $\lambda$ cannot be contained in any ray.
  Thus each ray in the Macdonald tree must pass through the points $(2^k)$ or $(1^{2^k})$ for each $k$, and so must be either $\{(n)\}_{n=1}^\infty$ or $\{(1^n)\}_{n=1}^\infty$, as claimed.
\end{proof}

\section{Recursive description of the Macdonald tree}
\label{sec:recursive}
\begin{figure}[h]
  \centering
  \includegraphics[scale = 0.62]{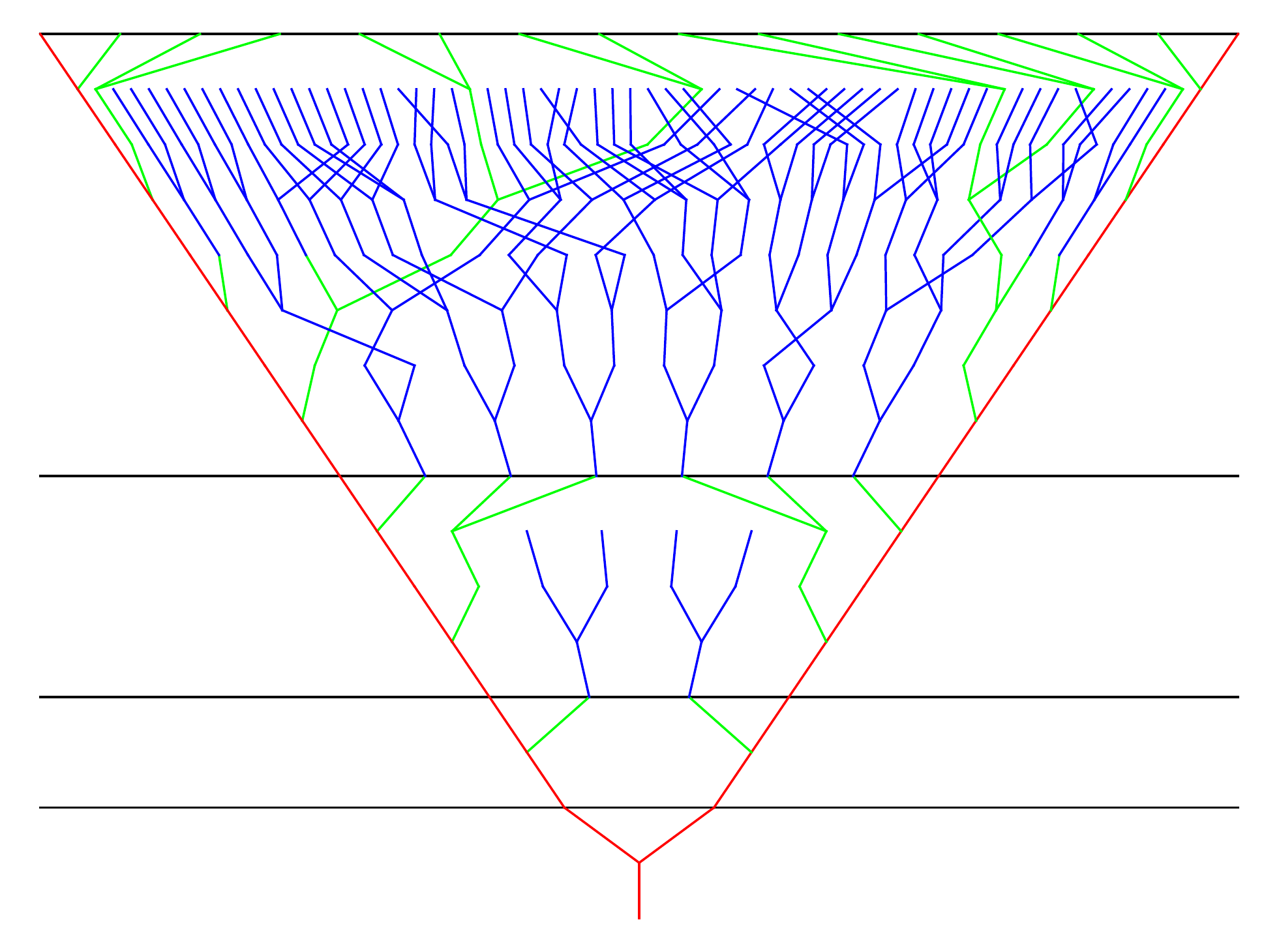}
  \caption{Hooks (joined by red and green edges) in the Macdonald tree}
  \label{fig:macdonald_tree}
\end{figure}
The Macdonald tree, when viewed as an abstract rooted binary tree, has a very short recursive description.
Let $T_k = \emptyset^{+[0, 2^k-1]}$, the subtree formed by nodes which are at distance at most $2^k-1$ from the root.
The recursive description allows for the construction of $T_{k+1}$ from $T_k$.

As a first step, construct a tree $\tilde T_k$ by adding one node $*$ to the root of $T_k$ (thus $\tilde T_k$ may be viewed as a rooted tree with root $*$).
Note that, for $k\geq 2$, $T_k$ has $2^{\binom k2}$ partitions of $2^k-1$, of which $2^{k-1}$ are hooks.
Among these hooks, there are two partitions $\lambda$ with $f_\lambda = 1$ (we will call these one-dimensional partitions).

\begin{quote}
  For $k\geq 2$, the tree $T_{k+1}$ is obtained from $T_k$ by attaching two branches to each partition of $2^k-1$ in $T_k$ that is a hook, each of these branches being isomorphic to $\tilde T_k$.
\end{quote}

Thus the nodes of $T_{k+1}$ can be partitioned into $2^k + 1$ subsets, in such a way that the induced subgraph on each of these subsets is isomorphic to $T_k$.
In order to be able to extend the recursive process further, we need to mark the hooks and one-dimensional partitions of $2^{k+1}-1$ in $T_{k+1}$.
To mark the hooks, choose one branch descending from each of the one-dimensional partitions in $T_k$.
This branch is, by construction, isomorphic to a copy of $T_k$.
The partitions in this branch which correspond to hooks of $T_k$ are the partitions of $2^{k+1}-1$ in $T_{k+1}$ which are hooks.
To mark partitions of $2^{k+1}-1$ which are one-dimensional, simply choose one of the one-dimensional partitions of $2^k-1$ in each branch.

These markings, of course, are only defined up to an automorphism of $T_{k+1}$ (which is not a problem, because we are only describing the Macdonald tree as an abstract binary tree).
The Macdonald tree up to partitions of $16$ is shown in Figure~\ref{fig:macdonald_tree}.
The green and red edges are the ones which join hooks.
Among these, the red edges join one dimensional partitions.
All remaining edges are coloured blue.
The horizontal lines mark powers of two.

\section{Hooks and Pascal's triangle}
\label{sec:pascal}
Among all partitions, consider the hooks, which are partitions of the form $(n_1+1, 1^{n_2})$ for nonnegative integers $n_1$ and $n_2$.
Pascal's triangle may be regarded as the Hasse diagram of the poset of pairs of non-negative integers, with $(n_1,n_2)\leq (m_1,m_2)$ if and only if $n_1\leq m_1$ and $n_2\leq m_2$.
The unique minimal element of this poset is $(0, 0)$, and it is graded with $(n_1,n_2)$ having rank $n_1 + n_2$.
In this poset, $(n_1, n_2)$ is covered by two elements, $(n_1+1, n_2)$ and $(n_1, n_2+1)$.
The number of saturated chains from $(0, 0)$ to $(n_1,n_2)$ is the binomial coefficient $\binom{n_1 + n_2}{n_1}$.
The subgraph of Pascal's triangle consisting of pairs $(n_1, n_2)$, where $\binom{n_1+n_2}{n_1}$ is odd  (see Figure~\ref{fig:sierpinski}) is closely related to the Sierpi\'nski triangle \cite{reiter1993determining}.
\begin{figure}[h]
  \centering
  \includegraphics[scale = 0.5]{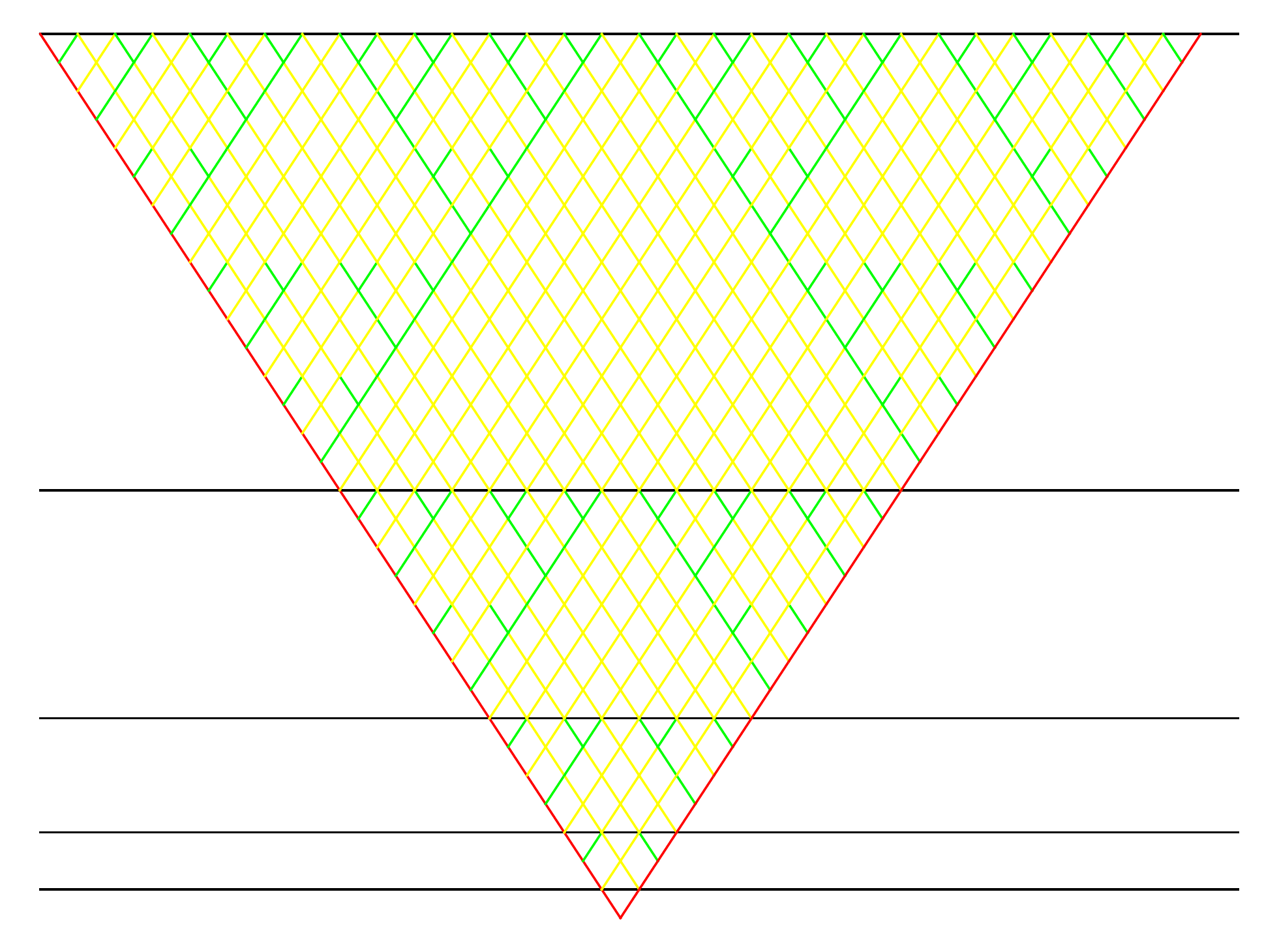}
  \caption{Odd binomial coefficients (joined by green and red edges) in Pascal's triangle up to $n=32$}
  \label{fig:sierpinski}
\end{figure}

The map $(n_1,n_2)\mapsto (n_1+1, 1^{n_2})$ is an embedding of Pascal's triangle into Young's lattice (with a shift of one in rank).
The image, consisting of all the hooks in Young's lattice, taken together with the empty partition $\emptyset$, is an order ideal in $\Lambda$, which we will denote by $P$.
Thus if $\lambda\in \Lambda$ is a hook, then every saturated chain from $\emptyset$ to $\lambda$ in $\Lambda$ is also a saturated chain in $P$.
This gives an amusing proof of the well-known formula:
\begin{displaymath}
  f_{(n_1+1, 1^{n_2})} = \binom{n_1+n_2}{n_1}.
\end{displaymath}
It follows that the subgraph induced in Young's graph by the set of odd-dimensional hooks is isomorphic to the graph of odd binomial coefficients in Pascal's triangle.
These are the green and red edges in Figures~\ref{fig:macdonald_tree} and ~\ref{fig:sierpinski}.

\section{Proof of Lemma~\ref{lemma:unique_hook}}
\label{sec:append-result-from}
Our proof of Lemma~\ref{lemma:unique_hook} relies on the following result of Frame, Robinson, and Thrall \cite[Lemma~2]{frame1954hook}.
\begin{lemma}
  \label{sec:macd-lemma-1}
  There exists a bijection from the set of cells in $\quot\lambda p$ onto the set of cells in $\lambda$ whose hook-lengths are divisible by $p$
  under which a cell of hook-length $h$ in $\quot\lambda p$ maps to a cell of hook-length $p h$ in $\lambda$.
\end{lemma}
  For a partition $\lambda$ of $n$, let $\alpha = \core\lambda 2$ and $(\mu^0,\mu^1) = \quot\lambda 2$.
  Let $a = |\alpha|$, $m_0 = |\mu^0|$, and $m_1 = |\mu^1|$ (so $n = a + 2m_0 + 2m_1$).
\begin{lemma}
  \label{sec:recurse-f}
  The partition $\lambda$ is odd if and only if $a\leq 1$ (so $\alpha$ is $\emptyset$ or $(1)$), the sets of place values where $1$ appears in the binary expansions of $a$, $2m_0$ and $2m_1$ are disjoint, and $\mu^0$ and $\mu^1$ are odd.
\end{lemma}
\begin{proof}
  By the hook-length formula \cite[Theorem~1]{frame1954hook}, we have
  \begin{equation}
    \label{eq:3}
    v_2(f_\lambda) = v_2(n!) - v_2(H_\lambda),
  \end{equation}
  where $H_\lambda$ is the product of all the hook-lengths of $\lambda$.
  By a well-known theorem of Legendre, $v_2(n!) = n - \nu(n)$, where $\nu(n)$ is the number of times $1$ occurs in the binary expansion of $n$.
  Moreover, by Lemma~\ref{sec:macd-lemma-1}, $v_2(H_\lambda) = m_0 + m_1 + v_2(H_{\mu^0}) + v_2(H_{\mu^1})$.
  Using these facts, the identity \eqref{eq:3} can be rewritten as
  \begin{displaymath}
    v_2(f_\lambda) = [a - \nu(a)] + [\nu(a) + \nu(2m_0) + \nu(2m_1) - \nu(n)] +  v_2(f_{\mu^0}) + v_2(f_{\mu^1}).
  \end{displaymath}
  The right-hand side is a sum of four nonnegative parts.
  It  is zero if and only if each part is zero.
  This happens only under the conditions of the lemma.
\end{proof}
\begin{proof}
  [Proof of Lemma~\ref{lemma:unique_hook}]
  Suppose that $\lambda$ is an odd partition of $n$ with $n>1$.
  Choose $k$ such that $2^k\leq n<2^{k+1}$.
  By Lemma~\ref{sec:recurse-f}, both $f_{\mu^0}$ and $f_{\mu^1}$ are odd, and exactly one of $m_0$ and $m_1$ satisfies $2^{k-1}\leq m_i<2^k$.
  By induction, this $\mu^i$ has a unique $2^{k-1}$-hook, while the other has none.
  By Lemma~\ref{sec:macd-lemma-1}, $\lambda$ has a unique $2^k$-hook.

  Moreover, by \cite [Theorem~4]{frame1954hook}, we have
  \begin{align}
    \label{eq:1}
    \core{\core\lambda{2^k}}{2} & = \alpha,\\
    \label{eq:2}
    \quot{\core\lambda{2^k}}{2} & = (\core{\mu^0}{2^{k-1}}, \core{\mu^1}{2^{k-1}}).
  \end{align}
  If $2^{k-1}\leq m_i< 2^k$, then $f_{\core{\mu^i}{2^{k-1}}}$ is odd by induction.
  Otherwise, $m_i<2^{k-1}$, so again $\core{\mu^i}{2^{k-1}} = \mu^i$, so $\core{\mu^i}{2^{k-1}}$ is odd.
  Thus $\core\lambda{2^k}$ is odd by Lemma~\ref{sec:recurse-f}.

  For the converse, suppose that $\lambda$ has a unique $2^k$-hook, and that $\core\lambda{2^k}$ is odd.
  The first condition implies that $|\core\lambda{2^k}| = n - 2^k$.
  Let $m'_i = |\core{\mu^i}{2^{k-1}}|$.
  We have
  \begin{displaymath}
    n - 2^k = a + 2m'_1 + 2m'_2,
  \end{displaymath}
  and the place values of the $1$s in the binary expansion of the summands on the right-hand side are disjoint.
  Since $2^{k-1}\leq m_i$ for exactly one $i$, \eqref{eq:2} implies that $m'_i = m_i$ for exactly one $i$ (say, $i=0$), and thus, for the other value of $i$ (say $i=1$), $m'_i = m_i - 2^{k-1}$.
  Thus $\mu^0$ has a unique $2^k$-hook and (by Lemma~\ref{sec:recurse-f} applied to $\core\lambda{2^k}$) $f_{\core{\mu^0}{2^{k-1}}}$ is odd.
  Thus, by induction on $k$, $f_{\mu^0}$ is odd.
  Since $m_1<2^{k-1}$, $\mu^1 = \core{\mu^1}{2^{k-1}}$, so (again by Lemma~\ref{sec:recurse-f} applied to $\core\lambda{2^k}$) $f_{\mu^1}$ is odd.
  Finally, the application of Lemma~\ref{sec:recurse-f} to $\lambda$ shows that $f_\lambda$ is odd.

  Since the odd partitions of $n$ have a unique $2^k$-hook, the partition $(1)$ occurs in their $2^k$-quotients once, and the partition $\emptyset$ occurs $2^k-1$ times.
  Thus there are $2^k$ possibilities for the $2^k$-quotient of such a partition once its $2^k$-core is fixed.
  Since a partition is determined by its core and quotient, the second assertion of the lemma follows.
\end{proof}

\section{Concluding remarks}
\label{sec:concluding-remarks}
In this article, we have described how Macdonald's enumerative results on odd partitions are reflected in Young's lattice.
The enumerative result in Macdonald's paper \cite{macdonald1971degrees} is a simple special case of his more general result on the enumeration of partitions $\lambda$ for which $f_\lambda$ is not divisible by a prime number $p$.
It would be interesting to see how these more general enumerative results are reflected in Young's lattice.
However, this can not be achieved by using only the methods here.
For instance, Lemma~\ref{lemma:unique_hook} does not hold for $p>2$.
Also, the partitions $\lambda$ with $f_\lambda$ coprime to $3$ do not form a tree (the partitions $(2)$ and $(1,1)$ both cover $(2,1)$; all three partitions have dimension coprime to $3$).

Another promising direction of generalization is to replace Young's lattice by an arbitrary $1$-differential poset.
Besides Young's lattice, the best-known example of a $1$-differential poset is the Young--Fibonacci lattice (see \cite[Section~5]{sta-diffpo}), denoted $Z(1)$.
For each $x\in Z(1)$, let $f_x$ denote the number of saturated chains in $[\hat 0, x]$.
Using the construction of $Z(1)$ by reflection extension \cite[Section~2.2]{roby-thesis}, it is easy to prove the following.
\begin{theorem}
  The subgraph induced in the Hasse diagram of $Z(1)$ by the set of elements $x\in Z(1)$ for which $f_x$ is odd is a binary tree where every element of even rank has one branch and every element of odd rank has two branches.
\end{theorem}
The following analogue of Macdonald's enumerative result is an immediate corollary.
\begin{theorem}
  The number of elements $x$ in $Z(1)$ of rank $n$ with $f_x$ odd is $2^{\lfloor n/2\rfloor}$.
\end{theorem}
For the Fibonacci $r$-differential poset with $r>1$, the subgraph induced in its Hasse diagram by the subset of elements $x$ with an odd number of saturated
chains in $[0, x]$ is a rooted tree if and only if $r$ is even. In this tree
every node has $r$ branches.

\subsection*{Acknowledgements}
This research was driven by computer exploration using the open-source
mathematical software \texttt{Sage}~\cite{sage} and its algebraic
combinatorics features developed by the \texttt{Sage-Combinat}
community~\cite{Sage-Combinat}.
A. A. was supported in part by a UGC Centre for Advanced Study grant.
The authors thank an anonymous referee for some helpful comments.

\end{document}